\documentclass[11pt,english]{article}
\usepackage{amsmath, amsthm, latexsym, amssymb,url}
\usepackage{amsfonts}
\usepackage{xypic}
\usepackage{epsfig}

\usepackage{graphicx}
\usepackage{yhmath}
\usepackage{mathdots}
\usepackage{pifont}

\input xy
\xyoption{all}

\newcommand{\shrinkmargins}[1]{
  \addtolength{\textheight}{#1\topmargin}
  \addtolength{\textheight}{#1\topmargin}
  \addtolength{\textwidth}{#1\oddsidemargin}
  \addtolength{\textwidth}{#1\evensidemargin}
  \addtolength{\topmargin}{-#1\topmargin}
  \addtolength{\oddsidemargin}{-#1\oddsidemargin}
 \addtolength{\evensidemargin}{-#1\evensidemargin}
  }

\shrinkmargins{.7}

\theoremstyle{plain}

\newtheorem{theorem}{Theorem}[section]
\newtheorem{corollary}[theorem]{Corollary}
\newtheorem{lemma}[theorem]{Lemma}
\newtheorem{proposition}[theorem]{Proposition}
\newtheorem{question}[theorem]{Question}
\newtheorem*{propo}{Proposition}
\newtheorem*{teo}{Theorem}

\newtheorem{definition}[theorem]{Definition}

\theoremstyle{remark}
\newtheorem{remark}[theorem]{Remark}

\theoremstyle{definition}

\def \Z { \mathbb{Z}}
\def \Q { \mathbb{Q}}

\def \R { \mathbb{R}}

\def \tr { \text{tr}}

\begin{document}

\thispagestyle{empty}
\setcounter{tocdepth}{7}

\title{The $(\alpha, \beta)-$ramification invariants of a number field.}
\author{Guillermo Mantilla-Soler}

\date{}

\maketitle

\begin{abstract}
Let $L$ be a number field. For a given prime $p$ we define integers $\alpha_{p}^{L}$ and $\beta_{p}^{L}$ with some interesting arithmetic properties. For instance,  $\beta_{p}^{L}$ is equal to $1$ whenever $p$ does not ramify in $L$ and  $\alpha_{p}^{L}$ is divisible by $p$ whenever $p$ is wildly ramified in $L$. The aforementioned  properties, although interesting, follow easily from definitions; however a  more interesting application of these invariants is the fact that they completely characterize the Dedekind zeta function of $L$. Moreover, if the residue class mod $p$ of $\alpha_{p}^{L}$ is not zero for all $p$ then such residues determine the genus of the integral trace.
\end{abstract}

 \section{ Introduction}

Let $L$ be a number field and let $p$ be a prime. In this paper we define three arithmetic invariants of $L$ attached to the prime $p$. The first two \[\alpha_{p}^{L} \ {\rm and} \ \beta_{p}^{L}\] are positive integers and the third $\mathfrak{a}_{p}^{L}$ is an integral quadratic form. The integers $\alpha_{p}^{L}$ and $\beta_{p}^{L}$, which we call {\it first and second ramification invariants},  capture the notion of ramification of the prime $p$ in $L$, furthermore they dictate if $p$ is wildly ramified or totally split in $L$:

\begin{propo}[Cf. Proposition \ref{ElLemaSplit}]

Let $L$ be a number field and let $p$ be rational prime. Then,

\begin{itemize}

\item[(a)] The prime $p$ does not ramifiy in $L$ if and only if $\beta_{p}^{L}=1$.

\item[(b)] The prime $p$ is totally split in $L$ if and only if  $\alpha_{p}^{L}=1.$

\item[(c)] The prime $p$ is wildy ramified in $L$ if and only if  $\alpha_{p}^{L}\equiv  0  \pmod{p.}$

\end{itemize}
\end{propo}

Although the above characterization is neat the first and second ramification invariants capture a lot more about the arithmetic of the number field than only a recipe for ramification. Recall that two number fields are called {\it arithmetically equivalent} if and only their Dedekind zeta functions coincide. As it turns out the knowledge of these invariants determines the arithmetic equivalence of the field, more precisely:

\begin{teo}[Cf. Theorem \ref{MainAE}]
Let $K,L$ be number fields of the same degree over $\Q$. The following are equivalent: 

\begin{itemize}
\item For almost all $p$  \[\alpha_{p}^{K} =\alpha_{p}^{L.}\] 

\item K and L are arithmetically equivalent.

\end{itemize}
\end{teo}

As we can see in Lemma \ref{AEyLos(a,b)} the {\it for almost all} in the statement above can not be replaced by {\it for all}. In other words the $\alpha_{p}^{L}$ invariants, runing over all primes, create a properly finer arithmetic invariant than the Dedekind zeta function $\zeta_{L}(s)$. Since the function $\zeta_{L}(s)$ refines several other arithmetic quantities e.g., the unit group, the discriminant, the Galois closure, the signature etc., see \cite[\S2]{Manti3}, it follows that the ordered set $\{\alpha_{p}^{L}\}$ is a new invariant with strong arithmetic implications.

\paragraph{Integral trace}

It is important to realize where do these invariants come from, and they are not just only an arbitrary combination of ramification and residue degrees. They appear naturally when trying to understand the local behavior of the integral trace; one sees them in the Jordan decomposition of the integral trace form over the $p$-adic integers. Moreover, in the absence of wild ramification the first ramification invariant determines the local structure of the trace: 

\begin{teo}[Cf. Theorem \ref{LaTrazaAca}] Let $K, L$ be two number fields with the same degree. Let $p$ be an odd prime and suppose that the discriminant of $K$ is equal to that of  $L$ up to squares in $(\Z_{p})^{*}$. Moreover, suppose that $p$ is not wildly ramified in either $K$ or $L$. Then, the integral trace forms of $K$ and $L$ are isometric over $\Z_{p}$ if and only if \[ \left( \frac{\alpha_{p}^{K}}{p} \right)= \left( \frac{\alpha_{p}^{L}}{p} \right).\]

\end{teo}

In the past we have used these invariants for the purpose of the classification of the isometry class of the integral trace, see for instance \cite{Manti, Manti1}. In this paper we show that they have also some other applications as for example the ones given in Theorem \ref{MainAE} and Proposition \ref{ElLemaSplit}. 

\subsection{Notation and definitions}\label{prelim}

We summarize here the most important notation used in the paper.

\begin{itemize}

\item Throughout  the paper whenever we say rational prime we refer to the usual integer primes together with the prime at infinity which, as Conway does \cite[Chapter 15, \S4]{conway}, we denote by $-1$. Explicitly, a rational prime $p$ will denote an element of the set $\{2,3,5,7,...\} \cup \{-1\}$.

\item For $a \in \Q_{p}$ we have that $v_{p}(a)$ is the usual $p$-adic valuation if $p\neq -1$, and for $p=-1$ we have that $v_{-1}(a)={\rm sign}(a)$

\item Given $a_{1},..., a_{n}$ elements of a ring $R$ --which in practice will be a maximal order on a number field or a local field-- the $R$-isometry class of a quadratic form $a_{1}x_{1}^{2}+...+a_{1}x_{n}^{2}$ will be denoted by $\langle a_1,...,a_n\rangle$. Whenever there is a possible ambiguity in the ring of definition of an isometry between two quadratic forms we will write $\cong_{R}$ to make it clear that the forms are considered to be over $R.$

\item The form $\underbrace{\langle e_{1},...,e_{1},e_{1}(-1)^{f_{1}-1}, e_{1}(-u_{p})^{f_{1}-1}\rangle}_{f_{1}}$ is to be understood to be $\langle -e_{1}, -e_{1}u_{p} \rangle$ in the case $f_{1}=2$.

\item The isometry class of the binary integral quadratic form $2xy$ over $\Z_{2}$,  the {\it hyperbolic plane},  will be denoted by $\mathbb{H}$.

\end{itemize}

\section{The invariants}

Given $L$ a number field and $p$ a rational prime we denote by  $g_{p}^{L}$ the number of primes in $O_{L}$ lying over $p$. Furthermore, let us denote by
\[e_{p}^{L}:= \sum_{i=1}^{g_{p}^{L}}e_{i}(L) \ \mbox{and} \ f_{p}^{L}:=\sum_{i=1}^{g_{p}^{L}}f_{i}(L),\] where $e_{1}(L),..., e_{g_{p}^{L}}(L)$ are the ramification degrees of the prime $p$ in $L$,  with respective residue degrees $f_{1}(L),..., f_{g_{p}^{L}}(L)$. When the field $L$ is clear from the context we will denoted the ramification (resp, residue) degrees only by $e_{i}$ (resp, $f_{i}$) instead of $e_{i}(L)$ (resp, $f_{i}(L)$).

\begin{definition}

For all prime $p$  we define the integer $u_{p}$ as follows 
\[u_{p}=\begin{cases}
-1 & \mbox{ if $p=-1$,}\\
 \ ~ 5 & \mbox{ if $p=2$,}\\
 \min\limits_{u \in \Z} \left \{ u : \mbox{ $0< u <p$  } | \left( \frac{u}{p} \right)=-1 \right \} & \ \mbox{for every other $p$.}
\end{cases}\]
\end{definition}


\begin{definition}\label{RamInvar} 
Let $L$ be a number field and let $p$ be a prime. The first and second ramification factors of $p$ in $L$ are the integers defined by:
\begin{align*}
\alpha_{p}^{L} &:= \left(\prod_{i=1}^{g_{p}^{L}}  e_{i}^{f_{i}}\right) u_{p}^{(f_{p}^{L}-g_{p}^{L})} \\ \beta_{p}^{L} & :=   \left(  \left(-1\right)^{\sum_{i=1}^{g_{p}^{L}}\left(\left\lfloor{\frac{(e_{i} -1)}{2}}\right\rfloor f_{i}\right) } \right) \left(\prod_{i=1}^{g_{p}^{L}}  e_{i}^{f_{i}(e_{i}-1)}\right)  u_{p}^{(n-f_{p}^{L}-e_{p}^{L}+g_{p}^{L})}. 
\end{align*}
\end{definition} 

\begin{remark}
Notice that  $\alpha_{-1}^{L}=\beta_{-1}^{L}=2^{s_{L}}$ where $s_{L}$ is the number of complex embeddings of $L$
\end{remark}


\begin{proposition}\label{ElLemaSplit}

Let $L$ be a number field and let $p$ be rational prime. Then,

\begin{itemize}

\item[(a)] The prime $p$ does not ramifiy in $L$ if and only if $\beta_{p}^{L}=1$.

\item[(b)] The prime $p$ is totally split in $L$ if and only if  $\alpha_{p}^{L}=1.$

\item[(c)] The prime $p$ is wildy ramified in $L$ if and only if  $\alpha_{p}^{L}\equiv  0  \pmod{p.}$

\end{itemize}
\end{proposition}

\begin{proof}
This is clear from the definitions.
\end{proof}

It follows from (b) above that the set of first ramification invariants determines $L$ whenever $L/\Q$ is Galois. More generally we have:

\begin{theorem}\label{MainAE}
Let $K,L$ be number fields of the same degree over $\Q$. The following are equivalent: 

\begin{itemize}
\item[(a)] For almost all rational prime $p$  \[\alpha_{p}^{K} =\alpha_{p}^{L.}\] 

\item[(b)] K and L are arithmetically equivalent.

\end{itemize}
\end{theorem}

\begin{proof}
Since the equivalence is up to a set of primes of zero density  we may assume that we are dealing with unramified primes. Since for an unramified prime $p$ the number $f_{p}^{L}$ is equal to the degree of $L/\Q$ then, since $K$ and $L$ have the same degree, the equality $\alpha_{p}^{K} =\alpha_{p}^{L}$ for such $p$ is equivalent to $g_{p}^{K}=g_{p}^{L}.$ Hence the result follows from  \cite[Theorem 1.2]{Manti2}
\end{proof}

\begin{lemma}\label{AEyLos(a,b)}
Let $L$ be a number field. The ordered set $\mathcal{A}^{L}:=\{ \alpha_{p}^{L}\}$ consisting of the first ramification invariants of the field $L$ is a finer invariant than the Dedekind zeta function $\zeta_{L}(s)$. In other words there exists a pair of arithmetically equivalent number fields $K$ and $L$ such that $\mathcal{A}^{K} \neq \mathcal{A}^{L}$.

\end{lemma}

\begin{proof}

Thanks to Theorem \ref{MainAE} we see that if $\mathcal{A}^{K} = \mathcal{A}^{L}$ then $\zeta_{K}(s)=\zeta_{L}(s)$. On the other hand consider the number fields $K$ and $L$ defined respectively by the polynomials $f:=x^7 - 3x^6 + 4x^5 - 5x^4 + 3x^3 - x^2 - 2x + 1$ and $g:=x^7 - x^5 - 2x^4 - 2x^3 + 2x^2 - x + 4$. These two arithmetically equivalent fields give a negative answer to a question asked in \cite{PerStu} regarding A.E fields with different ramification indices (See \cite[Theorem 3.7]{Manti3}). A calculation shows that $\alpha_{2}^{K}= 100$ and $\alpha_{2}^{L}=200$. In particular, $\mathcal{A}^{K} \neq \mathcal{A}^{L}$. 
\end{proof}

\begin{remark}

In the example given above the prime $p=2$ is the only prime for which $\alpha_{p}^{K}$ and $\alpha_{2}^{L}$ differ. This follows since $K$ and $L$ are arithmetically equivalent fields of discriminant $2^{6}691^{2}$ and $\alpha_{691}^{K}=8=\alpha_{691}^{L}.$

\end{remark}

\subsubsection{The quadratic form}

The final invariant we define is an integral quadratic form. As we will see below this form is $p$-adically, at least in the absence of wild ramification, determined by the first ramification invariant. 

\begin{definition} 

Let $L$ be a number field and let $p$ be a prime, including $p=-1$. The integral quadratic form $\mathfrak{a}_{p}^{L}$ as follows: \[\mathfrak{a}_{p}^{L} := \underbrace{\langle e_{1},...,e_{1},e_{1}(-1)^{f_{1}-1}, e_{1}(-u_{p})^{f_{1}-1}\rangle}_{f_{1}} \oplus...\oplus \underbrace{\langle e_{g_{p}^{L}},...,e_{g_{p}^{L}},e_{g_{p}^{L}}(-1)^{f_{g_{p}^{L}}-1}, e_{g_{p}^{L}}(-u_{p})^{f_{g_{p}^{L}}-1}\rangle}_{f_{g_{p}^{L}}}.\] 
\end{definition}

Notice that $\alpha_{p}^{L}$ is equal to the determinant of the form $\mathfrak{a}_{p}^{L}$, in particular $\alpha_{p}^{L}$ is determined by the isometry class of $\mathfrak{a}_{p}^{L}$.  Conversely, in the absence of wild ramification the form $\mathfrak{a}_{p}^{L}$ is completely determined by $\alpha_{p}^{L}$ over the $p$-adic integers. More explicitly:

\begin{lemma}
Let $L$ be a number field and let $p$ be an odd prime. Suppose that $p$ is not wildly ramified in $L$. Then as $\Z_{p}$ quadratic forms we have \[\mathfrak{a}_{p}^{L} \cong \underbrace{\langle1,....,1,\alpha_{p}^{L} \rangle}_{f_{p}^{L}}.\] 
\end{lemma}

\begin{proof}
This follows from \cite[VIII, \S3, Lemma 3.4]{cassels}.
\end{proof}

\subsection{Relation to other arithmetic invariants}

As we have seen the $(\alpha, \beta)$ invariants form a finer invariant than the Dedekind zeta function, and in the absence of wild ramification they also characterize the genus of the integral trace. It is natural to wonder whether there is a relation to other known arithmetic invariants of number fields. The first invariant that comes to mind is the ring of Adeles. Since the ivariants defined here are written in terms of the residue and ramification degrees the following is an inmediate consequence of \cite[Lemma 7]{iwa}. See also \cite[Lemma 1]{komatsu}.

\begin{theorem}\label{AdelesYlos(a,b)}
Let $ K,L$ be number fields. Suppose that as topological rings \[\mathbb{A}_{K} \cong \mathbb{A}_{L}.\] Then, for every prime $p$ we have that  $\alpha_{p}^{K}=\alpha_{p}^{L}$ and $\beta_{p}^{K}=\beta_{p}^{L}$.
\end{theorem}

A natural question that comes to mind is:

\begin{question}
Does the converse to Theorem \ref{AdelesYlos(a,b)} hold?
\end{question}

It turns out that the answer is no. 

\begin{lemma}\label{SameAdelesNotAlphas}
There are number fields $K$ and $L$ such that $\mathcal{A}^{K}=\mathcal{A}^{L}$ and $\mathbb{A}_{K} \not\cong \mathbb{A}_{L}$.
\end{lemma}

\begin{proof}
Consider the number fields $K$ and $L$ defined respectively by the polynomials $x^8-3$ and $x^8-48$. By \cite[Theorem 1]{komatsu1} we have that $\mathbb{A}_{K} \not\cong \mathbb{A}_{L}$ and $\zeta_{K}(s)=\zeta_{L}(s)$. Since $K$ and $L$ are arithmetically equivalent $\alpha_{p}^{K}=\alpha_{p}^{L}$ for every unramified prime $p$ and for $p=-1$. Since the only ramified primes are $p=2,3$ it suffices to show that $\alpha_{p}^{K}=\alpha_{p}^{L}$ for such primes. A calculation shows that both primes are totally ramified in each field hence $\alpha_{2}^{K}=\alpha_{2}^{L}= 8=\alpha_{3}^{K}=\alpha_{3}^{L}$.
\end{proof}

In other words Lemmas \ref{AEyLos(a,b)}, \ref{SameAdelesNotAlphas} and Theorem \ref{AdelesYlos(a,b)} tell us that there are implications between the strength of the invariants \[\mathbb{A}_{K} \longrightarrow \mathcal{A}^{K} \longrightarrow \zeta_{K}(s)\] and that in general such implications can not be reversed. 

\subsection{Integral trace form and its local representation}

Let $L$ be a number field and let $O_{L}$ be its maximal order. The integral trace form is the integral quadratic form given by \begin{displaymath}
\begin{array}{cccc}
q_L : & O_L  &  \rightarrow & \Z  \\  & x & \mapsto  & \ \tr_{L/\Q}(x^2).
\end{array}
\end{displaymath}
A local description of the form $q_{L}$ was first obtained in the case of  tame abelian and of odd degree  $L$ by Maurer \cite{Mau}. A broad generalization to Maurer's work for tame number fields was obtained by Erez, Morales and Perlis. Suppose $p$ is a rational prime which is at worst tamely ramified in $L$. In \cite{emp} the authors find a Jordan decomposition of the integral trace form $q_L$ when viewed as a form over $\Z_{p}$. In terms of the invariants defined in this paper their result can be stated as follows:

\begin{theorem}\label{general}
Let $L$ be a degree $n$ number field.  Let $p$ be a rational prime which is not wildly ramified in $L$. If we denote by $q_{L} \otimes \Z_{p}$ the quadratic form over $\Z_{p}$ induced by  $q_{L}$, then 
\[q_{L} \otimes \Z_{p} \cong 
\begin{cases}
\mathfrak{a}_{p}^{L}  \bigoplus p \otimes \underbrace{( \mathbb{H} \oplus...\oplus \mathbb{H})}_{\frac{n-f_{p}^{L}}{2}} & \mbox{ if $p=2$,}\\
\mathfrak{a}_{p}^{L}  \bigoplus p \otimes \underbrace{\langle 1,...,1,\beta_{p}^{L}\nu_{p}^{L} \rangle}_{n-f_{p}^{L}}  &\mbox{ if $p \neq 2$}. \\
\end{cases}\]
(See Definition \ref{thirdram} for the value of the $p$-adic integer $\nu_{p}^{L}$.)
\end{theorem}

As we mentioned above Theorem \ref{general} is not  original to us, and even though most of the tools we use in our proof differ from the ones in \cite{emp}, we follow their main strategy to prove Theorem \ref{general}. We have made all of the calculations very explicit to keep track of how the ramifications we defined came into existence. A particular instance where our approach is different to the one in \cite{emp} is the case $p=2$. Furthermore, the element $\nu_{p}^{L}$ is also new in our presentation since in the original result a small error lead to such term to disappear.

\section{The origings of the $(\alpha, \beta)$ invariants.}

To see where these invariants come from we begin by reviewing  the local integral trace.\\

\noindent The core idea here is to reduce the problem of localizing the integral trace at $p$ to calculating the integral trace of a finite extension of $\Q_{p}$. Suppose that $L$ is a number field and let $p$ be a rational prime. Then, the localization of the integral trace has an orthogonal decomposition as the sum of integral traces of finite extensions of $\Q_{p}$. More explicitly, if \[L \otimes_{\Q} \Q_{p} \cong K_{1} \times...\times K_{g}\] then  \[q_{L} \otimes \Z_{p} \cong \mathrm{tr}_{K_{1}/\Q_{p}}(x^{2})\mid_{O_{K_{1}}} \oplus ... \oplus \mathrm{tr}_{K_{g}/\Q_{p}}(x^{2})\mid_{O_{K_{g}}}.\] The above decomposition is easily obtained for the rational trace form, see \cite[Chapter II, (8.4)]{ne}, and from this decomposition one gets the above by comparing the discriminants of $O_{L}$ and $O_{K_{1}} \oplus ... \oplus O_{K_{g}}$;  see \cite[\S2]{Mau} for details. It follows that in order to determine  $q_{L} \otimes \Z_{p}$ it is enough to find, for any local field $K/\Q_{p}$, a Jordan decomposition of the $\Z_{p}$-integral trace form $\mathrm{tr}_{K/\Q_{p}}(x^{2})\mid_{O_{K}}$. Next we show how to do this for $K/\Q_{p}$ at worst tamely ramified.

\subsection{Local integral trace form}\label{localtrace}

Let $K/\Q_{p}$ be a finite extension with at worst tame ramification. To determine the $\Z_{p}$-integral trace form $\mathrm{tr}_{K/\Q_{p}}(x^{2})\mid_{O_{K}}$, we first find expressions for two integral traces depending on a convenient intermediate extension $F$.

\[\xymatrix{  
     & K \ar@{-}[d]^{\mathrm{tr}_{K/F}(x^{2})\mid_{O_{K}}}   & \\
          &  F \ar@{-}[d]^{\mathrm{tr}_{F/\Q_{p}}(x^{2})\mid_{O_{F}}} & \\
            & \Q_{p} &  \\
            }\]

By choosing $F$ carefully one can deduce the shape of the forms $\mathrm{tr}_{K/F}(x^{2})\mid_{O_{K}}$ and $\mathrm{tr}_{F/\Q_{p}}(x^{2})\mid_{O_{F}}$, and then recover $\mathrm{tr}_{K/\Q_{p}}(x^{2}) \mid_{O_{K}}$ by pasting together the above forms. The most natural choice for $F$ is $\Q_{p}^{un} \cap K$, the maximal unramified sub-extension of $K/\Q_{p}$. Since $K/\Q_{p}$ is tame, the extension $K/F$ is a totally ramified tame extension and such extensions have a very simple description. On the other hand, $F/{\Q_{p}}$ is the unique unramified extension of $\Q_{p}$ of degree $[F:\Q_{p}]$, moreover it is cyclic, hence it is also a very ``nice" extension.

\subsubsection{An intermediate extension}

Here we start by studying the trace form on the top extension, i.e., $\mathrm{tr}_{K/F}(x^{2})\mid_{O_{K}}$.

\begin{lemma} Let $K/F$ be a totally ramified tame extension of local fields over $\Q_{p}$. Let $e=[K:F]$. Then, there exist $\pi$ and $\pi_{F}$, uniformizers of $O_K$ and $O_F$, respectively, such that $O_{K}=O_{F}[\pi]$ and $\pi^{e}=\pi_{F}$.
\end{lemma}

\begin{proof}
This is a standard application of Eisenstein's criterion and Hensel's lemma. For details see \cite[Chapter 16]{Hasse}.
\end{proof}

\begin{lemma}\label{fact1} Let $F/\Q_{p}$ and unramified extension and let $K/F$ be a totally tamely ramified extension of degree $e$. Then, there exists an integral basis $\mathfrak{B}$ for $O_{K}$ over $O_{F}$, and a unit $\mu_{F} \in O_{F}^{*}$, such that the Gram matrix of $\mathrm{tr}_{K/F}(x^{2})\mid_{O_{K}} $ with respect the basis $\mathfrak{B}$ is  
\[M_{\mathfrak{B}}=\left[\begin{array}{ccccc}e & 0 & ... & ... & 0 \\0 & 0 & ... & 0 & e \cdot p \cdot \mu_{F} \\ \vdots &  &  & \iddots & 0 \\ \vdots & 0 & e\cdot p \cdot \mu_{F} & 0 & \vdots \\0 & e\cdot p \cdot\mu_{F} & 0 & ... & 0\end{array}\right]\]
\end{lemma}

\begin{proof}

Since $K/F$ is totally ramified and tame there exists $\pi$, a uniformizer for $O_{K}$, with minimal polynomial of the form $x^e-\pi_{F}$, where $\pi_{F}$ is a uniformizer of $O_{F}$, and such that $O_{K}=O_{F}[\pi]$. Since $F/\Q_{p}$ is unramified there exists $\mu_{F} \in O_{F}^{*}$ such that $\pi_{F}=\mu_{F} \cdot p$. Let $\zeta \in \mu_{e}(\overline{F})$ be a primitive root of $1$, and let $m$ be a non-negative integer. Then,
 \[ \mathrm{tr}_{K/F}(\pi^{m})= \pi^{m} (1+\zeta^{m}+(\zeta^{m})^{2}+...+(\zeta^{m})^{e-1})=
\begin{cases}
0,                  & \mbox{if $e\nmid m$,}\\
e\pi^{m},       & \mbox{if $e\mid m$.} \\
\end{cases}   \]
On the other hand if $0\leqslant i, j\leqslant e-1$ the only times that $i+j$ is a multiple of $e$ is whenever $i+j=0$ (i.e., $i=j=0$) or when $i+j=e$. Since $\pi^{e}=p\mu_{F}$ we have that 
\[\mathrm{tr}_{K/F}(\pi^{i} \cdot \pi^{j})=
\begin{cases}
e\cdot p \cdot \mu_{F}   & \mbox{$i+j=p$,}\\
e     & \mbox{$i=j=0$,} \\
0     & \mbox{otherwise.} \\
\end{cases}\] for all $0\leqslant i, j\leqslant e-1$. In other words $T_{\mathfrak{B}}$ is the Gram matrix of the trace in the basis $\mathfrak{B}=\{1,\pi,...,\pi^{e-1}\}$.

\end{proof}

\begin{theorem}\label{KoverF}
Let $F, K$ be as in Lemma \ref{fact1}.

\begin{itemize}
 \item[(a)] Suppose that  $p\neq 2$. Then, there exists a unit $\mu_{F} \in O^{*}_{F}$ such that the $O_{F}$ integral trace form $\mathrm{tr}_{K/F}(x^{2})\mid_{O_{K}} $ has the following diagonalization over $O_{F}$: \[\mathrm{tr}_{K/F}(x^{2})\mid_{O_{K}} \cong \left \langle e,p,p,...,p,p\mu_{F}^{e-1}e^{e-1}(-1)^{\left \lfloor\frac{e-1}{2}\right \rfloor} \right \rangle.\]
 
\item[(b)]  If $p=2$ \[ \mathrm{tr}_{K/F}(x^{2})\mid_{O_{K}} \cong \left \langle e \right \rangle \bigoplus  \left \langle 2 \right \rangle \otimes\underbrace{(\mathbb{H} \oplus...\oplus \mathbb{H})}_{\frac{e-1}{2}}.\] 
 
\end{itemize}

\end{theorem}

\begin{proof}
Let $q_{K/F}:=\mathrm{tr}_{K/F}(x^{2})\mid_{O_{K}}$. From Lemma \ref{fact1} we have that $q_{K/F} \cong \left \langle e, p\otimes T \right \rangle$ where $T$ is an $O_{F}$-form of discriminant $\mu_{F}^{e-1}e^{e-1}(-1)^{\left \lfloor\frac{e-1}{2}\right \rfloor}.$ 

\begin{itemize}
\item[(a)] Since $p \neq 2$, every form over $O_{F}$ is diagonalizable. Furthermore, every form over $O_{F}$ with discriminant $d \in O_{F}^{*}$ can be diagonalized to $\langle 1,...,1, d \rangle$. In particular, $T \cong \left \langle 1,...,1, \mu_{F}^{e-1}e^{e-1}(-1)^{\left \lfloor\frac{e-1}{2}\right \rfloor} \right \rangle$ and \[q_{F/K} \cong \left \langle e, p\otimes T \right \rangle \cong  \left \langle e,p,p...,p, p\mu_{F}^{e-1}e^{e-1}(-1)^{\left \lfloor\frac{e-1}{2}\right \rfloor} \right \rangle.\]

\item[(b)] If $p=2$ the form $T$ has Gram matrix equal to \[\left[\begin{array}{ccccc} 0 & ... & ... & 0 & e\mu_{F} \\ \vdots &  &  & \iddots & 0 \\ \vdots &  & e\mu_{F} &  & \vdots \\ 0 & \iddots &   &   & \vdots \\ e\mu_{F} & 0 & ... & ... & 0 \end{array}\right]\ \] i.e., \[T \cong 2e\mu_{F}x_{1}x_{e-1}+2e\mu_{F}x_{2}x_{e-2}+...+2e\mu_{F}x_{\frac{e-1}{2}}x_{\frac{e+1}{2}}.\] Since $e$ is odd we have that $e\mu_{F} \in O_{F}^{*}$. In particular, the transformation 
\begin{align*} x_{i} & \mapsto s_{i}(e\mu_{F})^{-1} \mbox{ for $1 \leq i \leq \frac{e-1}{2}$} \\  x_{j} & \mapsto t_{e-j}  \ \ \ \ \ \ \ \ \mbox{ for  $\frac{e+1}{2} \leq j \leq e-1$}
\end{align*}

shows that \[T \cong 2s_{1}t_{1}+2s_{2}t_{2}+...+2s_{\frac{e-1}{2}}t_{\frac{e-1}{2}}\cong \underbrace{\mathbb{H} \oplus...\oplus \mathbb{H}}_{\frac{e-1}{2}}.\]
 
\end{itemize}

\end{proof}

\subsubsection{From the intermediate extensions to the total extension}

Now that we know the Jordan decomposition of $\mathrm{tr}_{K/F}(x^{2})\mid_{O_{K}}$, we show how with this and with the Jordan decomposition of  $\mathrm{tr}_{F/\Q_{p}}(x^{2})\mid_{O_{F}}$, one can deduce the shape of $\mathrm{tr}_{K/\Q_{p}}(x^{2})\mid_{O_{K}}$.

\begin{lemma}\label{zinovy} Let $K/F$ be an extension of $p$-adic local fields of degree $e$.  Suppose that there are $\displaystyle \alpha_{0},..., \alpha_{e-1}\in O_{F}$ such that 

\[ \mathrm{tr}_{K/F}(x^{2})\mid_{O_{K}}\cong \begin{cases} \left \langle \alpha_{0},..., \alpha_{e-1} \right \rangle & \mbox{if $p \neq 2$, }\\  \left \langle \alpha_{0} \right \rangle \bigoplus \left( \left \langle 2 \right \rangle \otimes(\mathbb{H} \oplus...\oplus \mathbb{H}) \right) & \mbox{if $p=2$.} \end{cases}\] 

Then, we have that 

\[\mathrm{tr}_{K/\Q_{p}}(x^{2})\mid_{O_{K}}\cong \begin{cases} \bigoplus_{i=0}^{e-1} \mathrm{tr}_{F/\Q_{p}}^{\alpha_{i}}(x^{2})\mid_{O_{F}} & \mbox{ if $p\neq 2$,}\\ \mathrm{tr}_{F/\Q_{p}}^{\alpha_{0}}(x^{2})\mid_{O_{F}} \bigoplus \left( \left( \left \langle 2 \right \rangle \otimes(\mathbb{H} \oplus...\oplus \mathbb{H})  \right)\otimes \mathrm{tr}_{F/\Q_{p}}(x^{2})\mid_{O_{F}} \right) & \mbox{ if $p = 2$.} \end{cases}\]  

where $\mathrm{tr}_{F/\Q}^{\alpha_{i}} $ denotes the scaled trace form. In particular, we have  \[\mathrm{tr}_{K/\Q_{p}}(x^{2})\mid_{O_{K}}\cong  \begin{cases} \left( \langle \alpha_{0},...,\alpha_{e-2}\rangle \otimes \mathrm{tr}_{F/\Q_{p}}(x^{2})\mid_{O_{F}} \right ) \bigoplus \mathrm{tr}^{\alpha_{e-1}}_{F/\Q_{p}}(x^{2})\mid_{O_{F}} & \mbox{if $p \neq 2$,} \\ \left ( \langle \alpha_{0} \rangle \bigoplus \left( \left \langle 2 \right \rangle \otimes(\mathbb{H} \oplus...\oplus \mathbb{H}) \right ) \right) \bigotimes \mathrm{tr}_{F/\Q_{p}}(x^{2})\mid_{O_{F}} & \mbox{if $p=2$.} \end{cases}\] whenever $\displaystyle \alpha_{0},..., \alpha_{e-2} \in \Z_{p}$.

\end{lemma}

\begin{proof} Let $q_{K/F}:=\mathrm{tr}_{K/F}(x^{2})\mid_{O_{K}}$ and $q_{K}:=\mathrm{tr}_{K/\Q_{p}}(x^{2})\mid_{O_{K}}.$  

\begin{itemize}

\item[(a)] Let $p=2$. By hypothesis there is a basis $\{ w_{0}, w_{1}, \overline{ w_{1}},...,w_{\frac{e-1}{2}}, \overline{ w_{\frac{e-1}{2}}} \}$ for $O_{K}/O_{F}$ such that \[O_{K} \cong_{O_{F}} W_{0} \oplus ... \oplus W_{\frac{e-1}{2}}\]
where $W_{i}=:\mathrm{span}_{O_{F}} \{w_{i},\overline{ w_{i}}\}$ , $\overline{w_{0}} :=w_{0}$ and the direct sum is an orthogonal decomposition respect to the form $q_{K/F}$. Furthermore, $q_{K/F}\mid_{W_{0}} = \langle \alpha_{0} \rangle$ in the basis $\{w_{0}\}$ and $q_{K/F}\mid_{W_{0}}=\langle 2 \rangle \otimes\mathbb{H}$ in the basis $\{w_{i},\overline{ w_{i}}\}$. By the transitive property of the trace we have that the $\Z_{p}$-modules $W_{i}, W_{j}$ are orthogonal with respect to $q_{K}$.  Hence, \[q_{K} \cong_{\Z_{p}} q_{K}\mid_{W_{0}} \oplus ... \oplus q_{K}\mid_{W_{\frac{e-1}{2}}.}\] Moreover, $q_{K}\mid_{W_{0}} \cong \mathrm{tr}_{F/\Q_{p}}^{\alpha_{0}}(x^{2})\mid_{O_{F}}$. Indeed, let $\alpha =xw_{0}$ for some $x \in O_{F}$. Then $\mathrm{tr}_{K/\Q_{p}}((xw_{0})^{2})=\mathrm{tr}_{F/\Q_{p}}(\mathrm{tr}_{K/F}((xw_{0})^2)) = \mathrm{tr}_{F/\Q_{p}}(\mathrm{tr}_{K/F}(w_{0}^2)x^2)=\mathrm{tr}_{F/\Q_{p}}(\alpha_{0}x^{2}).$ Finally let $\{v_{1},...,v_{f}\}$ be a $\Z_{p}$-basis for $O_{F}$, and suppose $i>0$. Let $\{e_{1},...,e_{2f}\}$ be the  $\Z_{p}$-basis for $O_{K}$ given by $\{v_{1}w_{i},...,v_{f}w_{i},v_{1}\overline{w_{i}},...,v_{f}\overline{w_{i}} \}$. Using the transitive property again we have that \[ \mathrm{tr}_{K/\Q_{p}}(e_{k}e_{l})=0=\mathrm{tr}_{K/\Q_{p}}(e_{k+f}e_{l+f})\] and that \[\mathrm{tr}_{K/\Q_{p}}(e_{k+f}e_{l})=2\mathrm{tr}_{F/\Q_{p}}(v_{k}v_{l})\] for all $1 \leq k,l \leq f$. In particular,  $q_{K}\mid_{W_{i}} \cong (\langle 2 \rangle \otimes \mathbb{H}) \otimes \mathrm{tr}_{F/\Q_{p}}(x^{2})\mid_{O_{F}}$ from which the result follows.\\

\item[(b)] The case $p\neq 2$ follows exactly as in part (a) with $W_{0}$.

\end{itemize}

\end{proof}

\begin{corollary}\label{FirstStepTrace} Let $p \neq 2$ and let $K/\Q_{p}$ be a finite extension without wild ramification. Let  $F/\Q_{p}$ be the maximal unramified sub extension of $K$, and let $e$ be as usual. Then, there exists $\mu_{F} \in O_{F}^{*}$ such that 
\[\mathrm{tr}_{K/\Q_{p}}(x^{2})\mid_{O_{K}}\cong  \left(\underbrace{\left \langle e,p,...,p\right \rangle}_{e-1} \otimes \mathrm{tr}_{F/\Q_{p}}(x^{2})\mid_{O_{F}}\right) \bigoplus \left( \left \langle pe^{e-1}(-1)^{\left \lfloor\frac{e-1}{2}\right \rfloor}\right \rangle \otimes \mathrm{tr}^{\mu_{F}^{(e-1)}}_{F/\Q_{p}}(x^{2}) \mid_{O_{F}}\right).\] 
\end{corollary}

\begin{proof}
This follows immediately from Theorem \ref{KoverF} and Lemma \ref{zinovy}.
\end{proof}

\subsubsection{Maximal unramified extension}

Here we start by studying the trace form on the bottom extension, i.e., $\mathrm{tr}_{F/\Q_{p}}(x^{2})\mid_{O_{F}}$.

\begin{lemma}\label{odd}
Let $F/E$ be a degree $f$ cyclic Galois extension of discriminant $d$. Then, $d$ is a square if and only if $f$ is odd.
\end{lemma}

\begin{proof}
This is an elementary result which we leave to the reader.
\end{proof}

\begin{proposition}\label{SecondStepTrace}
Let $p\neq 2 $. Let $F/\Q_{p}$ be the unique unramified extension of degree $f$, and let $\alpha \in O_{F}^{*}$. Then, \[\mathrm{tr}^{\alpha}_{F/\Q_{p}}(x^{2}) \mid_{O_{F}} \cong \langle 1,...,1,{\rm N}_{F/\Q_{p}}(\alpha)u_{p}^{f-1}\rangle. \]  
\end{proposition}

\begin{proof}
Let $q_{F}^{\alpha}:=\mathrm{tr}^{\alpha}_{F/\Q_{p}}(x^{2}) \mid_{O_{F}}$. A simple calculation shows that disc$(q_{F}^{\alpha})=d\cdot{\rm N}_{F/\Q_{p}}(\alpha)$, where $d$ is the discriminant of $F$. Since $F/\Q_{p}$ is unramified and $\alpha \in O_{F}^*$ we have that ${\rm N}_{F/\Q_{p}}(\alpha)\cdot d \in \Z_{p}^{*}$, in particular \[q_{F}^{\alpha} \cong \langle 1,...,1, {\rm N}_{F/\Q_{p}}(\alpha) \cdot d\rangle.\] By Lemma \ref{odd} we have that $d=1$ if and only if $f$ is odd. Conversely, if $f$ is even $\Q_{p}(\sqrt{d})/\Q_{p}$ is the unique quadratic sub extension of $F/\Q_{p}$, thus $\Q_{p}(\sqrt{d})$ is the unique unramified quadratic extension of $\Q_{p}$ and it follows that $d=u_{p} \bmod (\Z_{p}^*)^2$. Summarizing  
\[d=
\begin{cases}
1,        &  \mbox{if $f$ is odd,}\\
u_{p} &  \mbox{if $f$ is even.} \\
\end{cases}\] i.e., $d=u_{p}^{f-1}$.
\end{proof}

\begin{remark}
Notice that by the same argument used above, we have that $d=u_{p}^{f-1}$  for $p=2$.
\end{remark}

\begin{proposition}\label{unra2}
Let $F/\Q_{2}$ be the unique unramified extension of degree $f$. Then, \[\mathrm{tr}_{F/\Q_{2}}(x^{2}) \mid_{O_{F}} \cong \langle 1,...,1,(-1)^{f-1}, (-u_{2})^{f-1}\rangle. \]  
\end{proposition}

\begin{proof}
Let $q:=\langle 1,...,1,(-1)^{f-1}, (-u_{2})^{f-1}\rangle$. The forms $q$ and $q_{F}:=\mathrm{tr}_{F/\Q_{2}}(x^{2}) \mid_{O_{F}}$ are both quadratic forms over $\Z_2$ with same discriminant, say $d$, and dimension. Moreover, thanks to \cite[Lemma 2.3]{conneryui}, we have that $q_{F}$ has Hasse-Witt invariant equal to \[(2,d)_{2}=(2,u_{2}^{f-1})_{2}=(2,5^{f-1})_{2}=(-1)^{f-1}.\] Since $q$ also has the same Hasse invariant, which follows from \[((-1)^{f-1},(-5)^{f-1})_{2}=(-1)^{f-1},\] the two forms can be considered as lattices inside a quadratic space over $\Q_{2}$. Since both forms have the same norm group, see \cite[Lemma 3]{Mau}, we have by  \cite[93:16]{Om}  that $q_{F} \cong q.$
\end{proof}

\subsubsection{Pasting it together}

Finally, using all of the above intermediate steps we can find a Jordan decomposition for $\mathrm{tr}_{K/\Q_{p}}(x^{2})\mid_{O_{K}}$.

\begin{theorem}\label{localgeneral}

Let $K/\Q_{p}$ be a finite extension without wild ramification, and let $e,f$ be the ramification and residue degrees of $K/\Q_{p}$. 

\begin{itemize}

\item[(a)]
If $p \neq 2$ there exits $\nu_{K} \in \Z_{p}^{*}$ such that 
\[\mathrm{tr}_{K/\Q_{p}}(x^{2}) \mid_{O_{K}} \cong \underbrace{\langle 1,...,1,e^{f}u_{p}^{f-1}\rangle}_{f} \bigoplus   \left \langle p \right \rangle \otimes \underbrace{\left \langle 1,...,1,(-1)^{f\left \lfloor\frac{e-1}{2}\right \rfloor}(e^{f}u_{p}^{f-1}\nu_{K})^{e-1} \right \rangle}_{f(e-1)}.\]

\item[(b)] If $p=2$ then, 
\[\mathrm{tr}_{K/\Q_{2}}(x^{2}) \mid_{O_{K}} \cong \underbrace{\langle e,...,e,e(-1)^{f-1}, e(-u_{2})^{f-1}\rangle}_{f} \bigoplus   \left \langle 2 \right \rangle \otimes \underbrace{\mathbb{H}\oplus...\oplus \mathbb{H}}_{\frac{f(e-1)}{2}}.\]

\end{itemize}

\end{theorem}

\begin{proof} Let $q_{K}:=\mathrm{tr}_{K/\Q_{p}}(x^{2}) \mid_{O_{K}}$ and let $F/\Q_{p}$ the maximal unramified sub extension of $K$. 

\begin{itemize}

\item[(a)] Suppose $p\neq 2$, and let $\nu_{K}:= N_{F/\Q_{p}}(\mu_{F})$ where $\mu_{F}$ is a unit in $O_{F}$ as in Lemma \ref{fact1}. Thanks to Corollary \ref{FirstStepTrace} and Proposition \ref{SecondStepTrace} we have that 
\[ q_{K} \cong \underbrace{\left \langle e,p,...,p\right \rangle}_{e-1} \otimes \underbrace{\left \langle 1,...,1,u_{p}^{f-1} \right \rangle}_{f} \bigoplus \left \langle p \right  \rangle \otimes \underbrace{\left \langle 1,...,1, e^{(e-1)f}(-1)^{\left \lfloor\frac{e-1}{2}\right \rfloor f} \nu_{K}^{e-1}u_{p-1}^{f-1}\right \rangle}_{f}\]

\begin{align*}
& \cong   \underbrace{\left \langle 1,...,1,e^{f}u_{p}^{f-1} \right \rangle}_{f} \bigoplus \left \langle p \right \rangle \otimes\underbrace{\left \langle 1,...,1,u_{p}^{(e-2)(f-1)} \right \rangle}_{f(e-2)} \bigoplus \left \langle p \right  \rangle \otimes \underbrace{\left \langle 1,...,1, e^{(e-1)f}(-1)^{\left \lfloor\frac{e-1}{2}\right \rfloor f} \nu_{K}^{e-1}u_{p-1}^{f-1}\right \rangle}_{f} \\
& \cong  \underbrace{\left \langle 1,...,1,e^{f}u_{p}^{f-1} \right \rangle}_{f} \bigoplus  \left \langle p \right \rangle \otimes \underbrace{\left \langle 1,...,1,(-1)^{f \left \lfloor\frac{e-1}{2}\right \rfloor}(e^{f}u_{p}^{f-1}\nu_{K})^{e-1}\right \rangle}_{f(e-1)}.
\end{align*}

\item[(b)] By Lemma \ref{aux2} any unimodular form $T$ over $\Z_{2}$ of dimension $f$ we have that \[\mathbb{H}\otimes T \cong \underbrace{\mathbb{H}\oplus...\oplus \mathbb{H}}_{f}.\] 
Thus, the case $p=2$ follows from Theorem \ref{KoverF},  Lemma \ref{zinovy},  and Proposition \ref{unra2}. 
\end{itemize}

\end{proof}

\begin{lemma}\label{aux2}
Let $T$ be a $\Z_{2}$-quadratic form of dimension $f$ such that disc$(T) \in \Z_{2}^{*}$. Then, \[\mathbb{H}\otimes T \cong \underbrace{\mathbb{H}\oplus...\oplus \mathbb{H}}_{f}.\]
\end{lemma}

\begin{proof}
Let $u$ be a unit in $\Z_{2}$. A calculation shows that the result is valid for $T \cong \langle u \rangle$, $T\cong 2x^2+2xy+2y^2$ and $T \cong \mathbb{H}$. Since every unimodular form over $\Z_{2}$ is sum of of these type of forms we are done.
\end{proof}

\subsubsection{Localizing the integral trace form}

Here we show how the Jordan decomposition of the localization of the trace form is obtained from the above results on the integral trace of a local field.\\

Let $L$ be a degree $n$ number field and let $p\neq 2$ be a prime at worst tamely ramified in $L$.  Let $L_{1},..., L_{g}$ be the completions of $L$ at the primes above $p$  i.e., the $p$-adic local fields defined by $L \otimes_{\Q} \Q_{p} \cong L_{1} \times...\times L_{g}$. Let $\nu_{i}:=\nu_{L_{i}}$ where $\nu_{L_{i}}$ is the element in $\Z_{p}^{*}/(\Z_{p}^{*})^{2}$ in Theorem \ref{localgeneral}.

\begin{definition}\label{thirdram} The third ramification factor of $p$ at $L$ is the element $\nu_{p}^{L} \in \Z_{p}^{*}/(\Z_{p}^{*})^{2}$ defined by
\[\nu_{p}^{L}:= \left(\prod_{i=1}^{g}  \nu_{i}^{(e_{i}-1)}\right).\]
\end{definition}

\begin{remark}
A priori $\nu_{p}^{L}$ seems to depend on the choice of uniformizers of each $L_{i}$, see Lemma \ref{fact1},  but as a consequence of the theorem below he have that ${\rm disc}(L)=p^{n-f_{p}^{L}}\alpha_{p}^{L}\beta_{p}^{L}\nu_{p}^{L} \pmod{ (\Z_{p}^{*})^{2}}$, so $\nu_{p}^{L}$ is well defined.
\end{remark}

\subsubsection{Proof of Theorem \ref{general}}

We recall the statement of the theorem. \\

{\it Let $L$ be a degree $n$ number field.  Let $p$ be a rational prime which is not wildly ramified in $L$. Then,
\[q_{L} \otimes \Z_{p} \cong 
\begin{cases}
\mathfrak{a}_{p}^{L}  \bigoplus p \otimes \underbrace{( \mathbb{H} \oplus...\oplus \mathbb{H})}_{\frac{n-f_{p}^{L}}{2}} & \mbox{if $p=2$,}\\
\mathfrak{a}_{p}^{L}  \bigoplus p \otimes \underbrace{\langle 1,...,1,\beta_{p}^{L}\nu_{p}^{L} \rangle}_{n-f_{p}^{L}}  & \mbox{if $p \neq 2$}. \\
\end{cases}\]
Furthermore, if $p \neq 2$ we have that $\mathfrak{a}_{p}^{L} \cong \underbrace{\langle1,....,1,\alpha_{p}^{L} \rangle}_{f_{p}^{L}}.$}
\begin{proof}
This follows from Theorem \ref{localgeneral} since $q_{L} \otimes \Z_{p} \cong \mathrm{tr}_{L_{1}/\Q_{p}}(x^{2})\mid_{O_{L_{1}}} \oplus ... \oplus \mathrm{tr}_{L_{g}/\Q_{p}}(x^{2})\mid_{O_{L_{g}}}$ where $L \otimes_{\Q} \Q_{p} \cong L_{1} \times...\times L_{g}$. The second assertion follows since  $\mathfrak{a}_{p}^{L}$ is unimodular over $\Z_{p}$.
\end{proof}

An interesting consequence of the above theorem 

\begin{theorem}\label{LaTrazaAca} Let $K, L$ be two number fields with the same degree. Let $p$ be an odd prime and suppose that the discriminant of $K$ is equal to that of  $L$ up to squares in $(\Z_{p})^{*}$. Moreover, suppose that $p$ is not wildly ramified in either $K$ or $L$. Then, the integral trace forms of $K$ and $L$ are isometric over $\Z_{p}$ if and only if \[ \left( \frac{\alpha_{p}^{K}}{p} \right)= \left( \frac{\alpha_{p}^{L}}{p} \right).\]

\end{theorem}

\begin{proof}

The result follows from Theorem \ref{LaTrazaAca} and \cite[Lemma 2.1]{Manti1}.

\end{proof}

\subsubsection{Interesting well known consequences}\label{oldies}

The following standard results can all be obtained as consequences of Theorem \ref{general}. 
\\

\begin{corollary}\ \\
\begin{itemize}

\item[(a)]  For every number field $L$ we have that  $\alpha_{-1}^{L}=\beta_{-1}^{L}=2^{s_{L}}$ and $f_{-1}^{L}=r_{L}+s_{L}$. Hence, in the case $p=-1$ the above formula is the well know result of O.Tauski \cite{Ta}:
\[ q_{L} \otimes \R  \cong  \langle 1,..., 1,2^{s_{L}} \rangle \oplus  (-1) \otimes \langle 1,...,1,2^{s_{L}} \rangle \cong  \underbrace{\langle 1,..., 1 \rangle}_{r_{L}+s_{L}} \oplus  \underbrace{\langle -1,...,-1 \rangle}_{s_{L}}.\]

\item[(b)] Since $\mathfrak{a}_{L}$ is unimodular, and $ \beta_{p}^{L}\nu_{p}^{L} \in \Z_{p}^{*} $, for all non wildly ramified prime $p$ in $L$ the well know formula \[v_{p}({\rm disc}(L))=n-f_{p}^{L}\]  is an immediate consequence of Theorem \ref{general}.\\

\item[(c)] If $p$ is a finite prime that is unramified in $L$ then Theorem \ref{general} implies the classic formula of Hasse,  \[\left( \frac{{\rm disc}(L)}{p}\right)=(-1)^{n-g_{L}^{p}}.\] \\
\end{itemize}
\end{corollary}  
  

\noindent
{\footnotesize Guillermo Mantilla-Soler, Department of Mathematics, Universidad Konrad Lorenz,\\
Bogot\'a, Colombia ({\tt gmantelia@gmail.com})}


\begin{thebibliography}{99}

\bibitem[Ca]{cassels} J.W. S. Cassels, \textit{Rational quadratic forms}, Dover publications, Inc., Mineola, NY, (2008).


\bibitem[C-P]{conner} P.E. Conner, R. Perlis, \textit{ A survey of trace forms of algebraic number
fields}, World Scientific, Singapore, 1984.

\bibitem[C-S]{conway} J.H. Conway, N.J.A. Sloane, \textit{Sphere packings, lattices and groups}, Third edition., Springer-Verlag New York, Inc. (1999).

\bibitem[C-Y]{conneryui} P.E. Conner, N. Yui, \textit{ The additive characters of the Witt ring of an algebraic number field}, Can. J. Math., Vol. XL, No. 3 (1988), 546-588.

\bibitem[EMP]{emp} B. Erez, J. Morales, R. Perlis \textit{Sur le Genre de la form trace},
Seminaire de Th\'{e}orie des Nombres de Bordeaux.  (Talence, 1987--1988), Exp.No. 18, 15 pp. 


\bibitem[Iw]{iwa} K. Iwasawa, \textit{On the rings of valuation vectors}, Ann. of Math. \textbf{57} (1953), 331-356.

\bibitem[Ko]{komatsu} K. Komatsu, \textit{On the adele rings of arithmetically equivalent fields}, Acta Arithmetica. \textbf{43} (1984) No 2, 93-95.


\bibitem[Ko1]{komatsu1} K. Komatsu, \textit{On the adele rings of algebraic number fields}, Kodai. Math. Sem. Rep \textbf{28} (1976), 78-84.

\bibitem[Ga]{galla} V.P. Gallagher, \textit{Local trace forms}, Linear and Multilinear Alg, \textbf{7} (1979), 167-174.

\bibitem[Ha]{Hasse} H. Hasse, \textit{Number theory}, Classics in Mathematics, Springer-Verlag Berlin Heidelberg, (2002).


\bibitem[Kr]{kru} M. Kr\"uskemper, \textit{Algebraic number field extensions with prescribed trace form}, Journal of Number Theory, \textbf{40} (1992), no. 1, 120-124.


\bibitem[Mau]{Mau} D. Maurer, \textit{The Trace-Form of an algebraic number field}, Journal of Number Theory, \textbf{5} (1973), 379-384.



\bibitem[Man]{Manti1} G. Mantilla-Soler, \textit{On the arithmetic determination of the trace,} Journal of Algebra \textbf{444} (2015), 1272-283.

\bibitem[Man1]{Manti} G. Mantilla-Soler, \textit{The Spinor genus of the integral trace form,} Transactions of the American Mathematical Society \textbf{369} (2017), 1547-1577.


\bibitem[Man2]{Manti2} G. Mantilla-Soler, \textit{An $\ell- p$ switch trick to obtain a new proof of a criterion for arithmetic equivalence,} Research in number theory \textbf{5} (2019) No 1, 1-5.


\bibitem[Man3]{Manti3} G. Mantilla-Soler, \textit{On a question of Perlis and Stuart regarding arithmetic equivalence,} To appear, New York Journal of Mathematics (2019).


\bibitem[Ne]{ne} J. Neukirch.  \textit{Algebraic Number Theory}, Springer, 1999.

\bibitem[PS]{PerStu} R. Perlis, D. Stuart \textit{A new characterization of arithmetic equivalence.} Journal of Number theory. \textbf{53} (1995), 300--308.

\bibitem[O]{Om} O.T.O'meara. \textit{Introduction to quadratic forms}, Die Grundlehren der mathematischen Wissenschaften, Vol. 117. Springer, Berlin;  Academic Press, New York; 1963. xi+342 pp

\bibitem[S2]{Serre2} J.P. Serre, \textit{Local fields}, Graduate Texts in Mathematics, \textbf{67}. Springer-Verlag, New York-Berlin, 1979. viii+241 pp.

\bibitem[Ta]{Ta} O. Taussky, \textit{The discriminant matrix of a number field}, J. London. Math. Soc. \textbf{43} (1968), 152-154.


\end{thebibliography}
\end{document}